\documentclass[12pt, oneside]{amsart}

\usepackage{amsmath,ifthen, amsfonts, amssymb, amsthm,
srcltx, 
   amsopn,tikz , url,
   }
\usetikzlibrary{decorations.markings,arrows}

\usepackage{graphicx}
\usepackage[shortlabels]{enumitem}
\usepackage{manfnt}

\newcommand{\showcomments}{yes}
\renewcommand{\showcomments}{no}

\newsavebox{\commentbox}
\newenvironment{com}%
{\ifthenelse{\equal{\showcomments}{yes}}%
{\footnotemark
        \begin{lrbox}{\commentbox}
        \begin{minipage}[t]{1.25in}\raggedright\sffamily\tiny
        \footnotemark[\arabic{footnote}]}
{\begin{lrbox}{\commentbox}}}%
{\ifthenelse{\equal{\showcomments}{yes}}%
{\end{minipage}\end{lrbox}\marginpar{\usebox{\commentbox}}}
{\end{lrbox}}}

\makeatletter
\newtheorem*{rep@thm}{\rep@title}
\newcommand{\newreptheorem}[2]{%
\newenvironment{rep#1}[1]{%
 \def\rep@title{#2 \ref{##1}}%
 \begin{rep@thm}}%
 {\end{rep@thm}}}
\makeatother

\newtheorem{thm}{Theorem}[section]
\newreptheorem{thm}{Theorem}
\newtheorem{lem}[thm]{Lemma}

\newtheorem{cor}[thm]{Corollary}

\newtheorem{prop}[thm]{Proposition}
\newtheorem*{thmno}{Illustrative theorem}

\theoremstyle{definition}
\newtheorem{defn}[thm]{Definition}
\newtheorem{rem}[thm]{Remark}

\newtheorem{construction}[thm]{Construction}

\DeclareMathOperator{\diam}{diam}
\DeclareMathOperator{\Comp}{Comp}
\DeclareMathOperator{\im}{im}
\DeclareMathOperator{\cd}{cd}

\newcommand{\p}{\textup{\textsf{p}}}

\newcommand{\dist}{\textup{\textsf{d}}}

\newcommand{\systole}[1]{\ensuremath{\| #1 \|}}
\newcommand{\nclose}[1]{\ensuremath{\langle\!\langle\, #1\,\rangle\!\rangle}}

\begin{document}

\title{
Hyperbolicity in non-metric cubical small-cancellation}

\author{Macarena Arenas}
\author{Kasia Jankiewicz}
\author[D.~T.~Wise]{Daniel T. Wise}
	\address{Clare College \\
            University of Cambridge \\
            United Kingdom  CB2 1TL and \\
            Centre for Mathematical Sciences\\
			University of Cambridge\\
			United Kingdom CB3 0WB}
	\email{mcr59@cam.ac.uk}
	\address{Department of Mathematics\\
			University of California\\
			Santa Cruz, USA 95064}
	\email{kasia@ucsc.edu}
	\address{Department of Mathematics \& Statistics\\
                    McGill University \\
                    Montreal, Quebec, Canada H3A 0B9}
           \email{wise@math.mcgill.ca}


\begin{abstract}
Given a non-positively curved cube complex $X$, we prove that the quotient of $\pi_1X$ defined by a cubical presentation $\langle X\mid Y_1,\dots, Y_s\rangle$ satisfying sufficient non-metric cubical small-cancellation conditions is hyperbolic provided that $\pi_1X$ is  hyperbolic. This  generalises  the fact that finitely presented classical $C(7)$ small-cancellation groups are hyperbolic.
 \end{abstract}

\maketitle

\section{Introduction}
A cubical presentation is a higher dimensional generalisation of a classical group presentation in terms of generators and relators. A non-positively curved cube complex $X$ plays the role of the ``generators'', and the ``relators" are local isometries of non-positively curved cube complexes $Y_i\hookrightarrow X$. The associated group is  the quotient of  $\pi_1 X$ by the normal closure $\nclose{\{\pi_1 Y_i\}}_{\pi_1X}$ of $\pi_1 Y_i$. As in the classical setting, this group is the fundamental group of $X$ with the $Y_i$'s coned off. 
Likewise, cubical small-cancellation theory, introduced in~\cite{WiseIsraelHierarchy}, is a generalisation of classical small-cancellation theory (see e.g.\ \cite{LS77}). In both the classical and cubical cases, the small-cancellation conditions are expressed in terms of pieces. A \emph{piece} in a classical presentation is a word that appears in two different places among the relators. 
The \emph{non-metric small-cancellation condition $C(p)$} where $p>1$ asserts that no relator is a concatenation of fewer than $p$ pieces.
The \emph{ metric small-cancellation condition $C'(\alpha)$} $\alpha\in (0,1)$ asserts that $|P|< \alpha |R|$ whenever $P$ is a piece in a relator $R$. 
Note that $C'(\frac{1}{p})\Rightarrow C(p+1)$. Pieces in cubical presentation are defined similarly, and the same implication holds in the cubical case.

Cubical small-cancellation has proven to be a fruitful tool in the study of groups acting on CAT(0) cube complexes. It was used by Wise as a step in his proof of the Malnormal Special Quotient Theorem~\cite{WiseIsraelHierarchy}, and as such, played a crucial role in the proofs of the Virtual Haken and Virtual Fibering conjectures~\cite{AgolGrovesManning2012}. Cubical presentations and cubical small-cancellation theory were also studied and utilised in~\cite{Arenas2023pi2, Arenas22,  ArzhantsevaHagen16, FuterWise2016, HuangWiseSpecial, Jankiewicz17,  JankiewiczWise18, PW18}. 
While classical small cancellation groups have virtual cohomological dimension $\leq 2$ ~\cite{Lyn66}, there exist cubical small cancellation groups with arbitrarily large virtual cohomological dimension, which is moreover controlled by $\cd\pi_1 X$ and $\cd\pi_1 Y_i$ ~\cite{Arenas2023pi}.

To illustrate the difference between metric and non-metric conditions, consider the following presentation: $\langle a,b \mid a^n w \rangle$. When $w$ is a long messy word (read: small cancellation) starting and ending in $b$, then the $C(p)$ condition holds for all $n$. However $a^{n-1}$ is a piece! So $C'(\alpha)$ fails for sufficiently large $n$.
Similar examples can be produced in the cubical setting. 
For instance, let $X= S \vee A$ where $S$ is a cubulated surface and $A$ is a circle. Let $w$ be a small-cancellation path in $X$ whose initial and terminal edges lie in $S$.  
Let $\widetilde w$ be the lift of $w$ to $\widetilde X$, and let $\widetilde W$ be the combinatorial convex hull of $\widetilde w$. 
Let $a^n$ be a length $n$ arc that immerses onto $A$.
Let $Y$ be the quotient of $a^n\cup \widetilde W$, identifying the endpoints of $a^n$ and  $\widetilde w$. Then $\langle X\mid Y\rangle$ is a cubical presentation satisfying $C(p)$, but not $C'(\frac 1 p)$, when $n\gg 0$. The pumping lemma shows that for any $X$ with $\pi_1X$ non-elementary hyperbolic, there are presentations $X^*$ that are non-metric small cancellation, but not metric small cancellation.

It is a fundamental result of classical small-cancellation theory that a group admitting a finite presentation satisfying the classical $C'(\frac 16)$ or $C(7)$ condition is hyperbolic. 
In analogy with the metric classical small-cancellation case, a cubical $C'(\frac {1} {14})$ presentation $ \langle X \mid Y_1,\dots, Y_s\rangle$ yields a hyperbolic group if   $\pi_1X$ is hyperbolic and the $Y_i$ are compact~\cite[Thm 4.7]{WiseIsraelHierarchy}. However, the proof of that result does not extend to the non-metric case.
The goal of this paper is to prove the following statement which  recovers and generalises the result from the $C'(\frac{1}{14})$ setting. 

\begin{repthm}{thm: hyperbolicity of C(p)}
Let $X^*= \langle X \mid Y_1,\dots, Y_s\rangle$ be a cubical presentation satisfying the $C(p)$ cubical small-cancellation condition for $p\geq 14$, 
where $X, Y_1, \dots, Y_s$ are compact, and $\pi_1X$ is hyperbolic. 
Then $\pi_1 X^*$ is hyperbolic.
\end{repthm}

\subsection{Proof strategy} \label{subsec:strategy}
The main ingredients of the proof are the notion of the piece metric (Definition~\ref{defn:piece metric}) and Papasoglu's thin bigon criterion for hyperbolicity (Proposition~\ref{prop:thin bigon criterion}).

The most immediate way of proving hyperbolicity for finitely presented $C'(\frac{1}{6})$ groups is to show that a linear isoperimetric inequality holds for their Cayley complexes. This follows from  the fact that $C'(\frac{1}{6})$ presentations satisfy Dehn's algorithm by Greendlinger's Lemma (see for instance~\cite[V.4.5]{LS77}). 
In the $C(7)$ setting, one is no longer guaranteed to have a Dehn presentation (consider the $\langle a,b | a^n w \rangle$ examples described above).
 Instead, the usual way of proving hyperbolicity in this generality relies on the combinatorial Gauss-Bonnet Theorem. 
Another way is to realise that reduced disc diagrams satisfy $C'(\frac{1}{6})$ if we regard all pieces as having length~$1$. In fact, this viewpoint leads to the piece-metric. 

To illustrate the basic idea behind our  strategy, we sketch a proof of hyperbolicity in the $C(7)$ case using the piece-metric and the thin bigon criterion.
The definition of the classical $C(n)$ condition and of all the diagrammatic notions introduced for cubical presentations in Section~\ref{sec:cubical presentations} can be particularised to this setting, and a version of Greendlinger's Lemma  also holds (see for instance~\cite[V.4.5]{LS77}).

\begin{thmno}  Let $X$ be a  $2$-complex satisfying the  $C(7)$ small-cancellation condition. Then $X$ is hyperbolic with the piece metric.
\end{thmno}

This implies hyperbolicity of finitely presented $C(7)$ groups, since in that case the piece-metric and the usual  combinatorial metric on the Cayley graph are quasi-isometric (see Proposition~\ref{prop: qi}).

\begin{proof}
We check that all bigons in $X^{(0)}$ are $\epsilon$-thin in the piece metric for some $\epsilon>0$, and apply Proposition~\ref{prop:thin bigon criterion} to conclude that $X$ is hyperbolic.

 Let $\gamma_1, \gamma_2$ be piece-geodesics forming a bigon in $X$, and let $D\rightarrow X$ be a reduced disc diagram with $\partial D=\gamma_1\bar\gamma_2$. We claim that $D$ is a (possibly degenerate) ladder, and hence that the bigon $\gamma_1, \gamma_2$ is $1$-thin, since by definition any two cells in a ladder intersect along at most $1$ piece.

 Indeed, by Greendlinger's Lemma, $D$ is either a ladder, or contains at least $3$ shells and/or spurs. First note that $D$ cannot have spurs, as these can be removed to obtain  paths $\gamma'_1, \gamma'_2$ with the same endpoints as $\gamma_1, \gamma_2$, and which are shorter in the piece metric, thus contradicting that $\gamma_1, \gamma_2$ are piece-geodesics.
 
 If $D$ has at least $3$ shells, then at least $1$ shell $S$ must have its outerpath  $Q$ contained in either $\gamma_1$ or $\gamma_2$. Since both cases are analogous, assume $Q \subset\gamma_1$, and  let $R$ be the innerpath of $S$. 
 Since $S$ is a shell and $X$ satisfies $C(7)$, then $R$ is the concatenation of at most $3$ pieces, so $|R|_\p < |Q|_\p$, and the path $\gamma_1'$ obtained from $\gamma_1$ by traversing $R$ instead of $Q$ is the concatenation of less pieces than  $\gamma_1$, contradicting that $\gamma_1$ is a piece-geodesic.
 
 Thus, $D$ is a ladder, and the proof is complete. 
\end{proof}

\subsection{Structure of the paper}
The paper is organised as follows. In Section 2, we give background on cube complexes, cubical group presentations, and cubical small-cancellation. In Section 3, we recall a criterion for hyperbolicity for groups acting on graphs. In Section 4, we define and analyse the piece metric. In Section~5, we prove Theorem~\ref{thm: hyperbolicity of C(p)}.

\subsection*{Acknowledgement}
We thank the referee for helpful comments. The first author was supported by a Cambridge Trust \& Newnham College Scholarship, and by the Denman Baynes Junior Research Fellowship. The second author was supported by the NSF grants DMS-2203307 and DMS-2238198. The third author was supported by NSERC.

\section{Cubical background}\label{sec:background}
\subsection{Non-positively curved cube complexes}
We assume that the reader is familiar with \emph{CAT(0) cube complexes}, which are CAT(0) spaces
having cell structures where each cell is isometric to a cube. We refer the reader to \cite{BridsonHaefliger, Sageev95, Leary_KanThurston, WiseIsraelHierarchy}. A \emph{non-positively curved cube complex} is a cell-complex $X$ whose universal cover $\widetilde X$ is a CAT(0) cube complex. A \emph{hyperplane}  $\widetilde H$ in $\widetilde X$ is a subspace
whose intersection with each $n$-cube $[0,1]^n$ is either empty or consists of the subspace where exactly one coordinate is restricted to $\frac12$.
For a hyperplane $\widetilde{H}$ of $\widetilde{X}$, we let $N(\widetilde{H})$ denote its \emph{carrier}, which is the union of all closed cubes intersecting $\widetilde{H}$. 
The \emph{combinatorial metric} $\dist$ on the $0$-skeleton of a non-positively curved cube complex $X$ is a length metric
where the distance between two points is the length of the shortest combinatorial path connecting them.
A map $\phi:Y\rightarrow X$ between non-positively curved cube complexes is a \emph{local isometry}
if $\phi$ is locally injective, $\phi$ maps open cubes homeomorphically to open cubes,
and whenever $a,b$ are concatenable edges of $Y$,
 if $\phi(a)\phi(b)$ is a subpath of the attaching map of a 2-cube of $X$,
 then $ab$ is a subpath of a 2-cube in $Y$.

\subsection{Cubical presentations}\label{sec:cubical presentations} We recall the notion of a cubical presentation, and the cubical small-cancellation conditions from \cite{WiseIsraelHierarchy}.

A \emph{cubical presentation} $\langle X \mid Y_1, \ldots, Y_m \rangle$ consists of a non-positively curved cube complex $X$, and a set of local isometries $Y_i \looparrowright X$ of non-positively curved cube complexes. We use the notation $X^*$ for the cubical presentation above. As a topological space, $X^*$ consists of $X$ with a cone on $Y_i$ attached to $X$ for each $i$. The vertices of the cones on $Y_i$'s will be referred to as \emph{cone-vertices} of $X^*$. The cellular structure of $X^*$ consists of all the original cubes of $X$, and the ``pyramids" over cubes in $Y_i$ with a cone-vertex for the apex. 

As mentioned in the introduction, cubical presentations generalise classical group presentations. Indeed, a classical presentation complex associated with a group presentation $G = \langle S \mid R\rangle$ can be viewed as a cubical presentation where the non-positively curved cube complex $X$ is just a wedge of circles, one corresponding to each generator in $S$. The complexes $Y_i$ correspond to relators $r_i$ in $R$. Each cycle $Y_i$ has length $|r_i|$, and the local isometry $Y_i\looparrowright X$ is defined by labelling the edges of $Y_i$ with the letters of $r_i$.

The universal cover $\widetilde{X^*}$ consists of a cube complex $\widehat X$ with cones over copies of $Y_i$'s. The complex $\widehat X$ is a covering space of $X$. A \emph{combinatorial geodesic} in $\widetilde{X^*}$ is a combinatorial geodesic in $\widehat X$, viewed as  a path in $\widetilde{X^*}$.

\subsection{Disc diagrams in $X^*$}

Throughout this paper, we will be analysing properties of \emph{disc diagrams}, which we introduce below together with some associated terminology: 

A map $f: X \longrightarrow Y$ between 2-complexes is \emph{combinatorial} if it maps cells to cells of the same dimension.
A complex is \emph{combinatorial} if all attaching maps are combinatorial, possibly after subdividing the cells. 

A \emph{disc diagram} is a compact, contractible $2$-complex $D$ with a fixed planar embedding $D\subseteq \mathbb S^2$. The embedding $D\hookrightarrow \mathbb S^2$ induces a cell structure on $\mathbb S^2$, consisting of the $2$-cells of $D$ together with an additional 2-cell, which is the 2-cell at infinity when viewing $\mathbb S^2$ as the one point compactification of $\mathbb R^2$. The \emph{boundary path} $\partial D$ of $D$ is the attaching map of the 2-cell at infinity.
Similarly, an \emph{annular diagram} is a compact $2$-complex $A$  with a fixed planar embedding $A\subseteq \mathbb S^2$ and the homotopy type of $\mathbb S^1$.  The annular diagram $A$ has two boundary cycles $\partial_{in} A$ $\partial_{out} A$. 
A \emph{disc diagram in $X^*$} is a combinatorial map $(D,\partial D)\to(X^*, X^{(1)})$ of a disc diagram. 
The $2$-cells of a disc diagram $D$ in $X^*$ are of two kinds: squares mapping onto squares of $X$, and triangles mapping onto cones over edges contained in $Y_i$. 
The vertices in $D$ which are mapped to the cone-vertices of $X^*$ are also called the \emph{cone-vertices}. 
Triangles in $D$ are grouped into cyclic families meeting around a cone-vertex. 
We refer to such families as \emph{cones}, and treat a whole such family as a single $2$-cell. 
A \emph{cone-cell} $C$ is the union of an annular square diagram $A\to D$ whose interior embeds in $D$, together with a cone over $\partial_{in}A$. See Figure~\ref{fig:cone-cell}.

\begin{figure}
\includegraphics[scale=0.4]{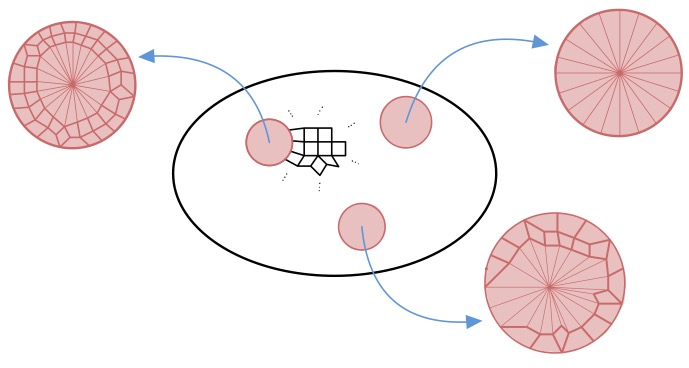}
\caption{Cone-cells in a disc diagram. In figures we will often omit the cell structure of cone-cells, unless needed.}\label{fig:cone-cell}
\end{figure}

We emphasize that this definition differs slightly from the definition of a cone-cell in the literature, where $A$ is simply a circle; allowing $A$ to be an arbitrary annular diagram, $D$ implicitly comes  equipped with a \emph{choice} of cone-cells.

The \emph{square part} $D_{\square}$ of $D$ is a subdiagram which is the union of all the squares that are not contained in cone-cells.

A \emph{square disc diagram} is a disc diagram whose square part is the whole diagram, i.e.\ it contains no cone-cells. A \emph{mid-interval} in a square, viewed as $[0,1]\times[0,1]$, is an interval $\{\frac 12\}\times [0,1]$ or $[0,1]\times \{\frac 1 2\}$. 
A \emph{dual curve} in a square disc diagram $D$ is a curve which intersect each closed square either trivially, or along a mid-interval, i.e., a dual curve is a restriction of a hyperplane in $X$ to $D$. We note that for each $1$-cube of $D$, there exists a unique dual curve crossing it~\cite[2e]{WiseIsraelHierarchy}.

The \emph{complexity} of a disc diagram $D$ in $X^*$ is defined as $$\Comp(D) = (\# \text{cone-cells}, \#\text{squares in $D_\square$}).$$ We say that $D$  has \emph{minimal complexity} if $\Comp(D)$ is minimal in the lexicographical order among disc diagrams with the same boundary path as $D$. A disc diagram $D$ in $X^*$ is \emph{degenerate} if $\Comp(D) = (0,0)$.
A disc diagram $D,$ in $X^*$ is \emph{singular} if $D$ is not homeomorphic to a closed ball in $\mathbb R^2$. This is equivalent to $D$ either being a single vertex or an edge, or containing a cut vertex. 
In particular, every degenerate disc diagram is singular.

A square $s$ is a \emph{cornsquare} on a cone-cell $C$ if a pair of dual curves emanating from consecutive edges $a, b$ of $c$ terminates on consecutive edges $a',b'$ of $\partial D$.

\begin{defn}[Reduction moves]\label{defn:reduction moves}
We define six types of reduction moves.
See Figure~\ref{fig:reduced}.
\begin{figure}
\includegraphics[scale = 0.38]{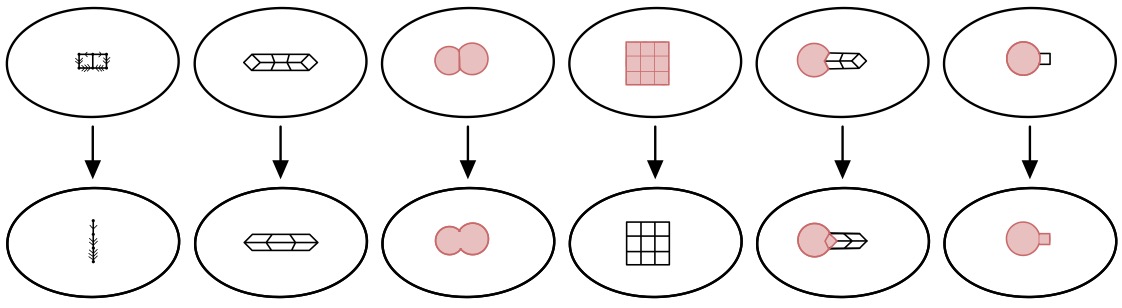}
\caption{The six reduction moves from Definition~\ref{defn:reduction moves}.}\label{fig:reduced}
\end{figure}
\begin{enumerate}[start=0]
\item \label{item:red0} Cancelling  a pair of squares $s,s'$ meeting at one edge $e$ in the disc diagram, whose map to $X^*$ factors through a reflection identifying them. That is, cutting out $e\cup Int(s)\cup Int(s')$ and then glueing together the paths $\partial s-e$ and $\partial s'-e$.
\item \label{item:red1} Replacing a minimal bigon-diagram, i.e.\ a disc subdiagram containing two dual curves intersecting each other twice, which is not contained in any other such subdiagram, with a lower complexity square disc diagram with the same boundary. 
\item \label{item:red3} Replacing a pair of adjacent cone-cells with a single cone-cell.
\item \label{item:red2} Replacing a cone-cell with a square disc diagram with the same boundary.
\item \label{item:red4} Absorbing a cornsquare $s$ to a cone-cell $C$, i.e.\ replace a minimal subdiagram containing $C$ and the two dual curves starting at $C$ and ending in $s$ with a lower complexity disc diagram with the same boundary and containing a cone-cell $C\cup s'$ for some square $s'$. 

\item \label{item:red5} Absorbing a square with a single edge in a cone-cell into the cone-cell. 
\end{enumerate}
\end{defn}

\begin{defn}[Reduced and weakly reduced disc diagram]
A disc diagram $D\to X^*$ in a cubical presentation is 
\begin{itemize}
\item \emph{reduced} if no moves \eqref{item:red0} - \eqref{item:red5} from Definition~\ref{defn:reduction moves} can be performed in $D$.
\item \emph{weakly reduced} if no moves \eqref{item:red1} - \eqref{item:red5} from Definition~\ref{defn:reduction moves} can be performed in $D$.
\end{itemize} \end{defn}
Note that if $D$ has minimal complexity then $D$ is reduced, and that, in particular, each reduction move outputs a diagram $D'$ with $Area(D')< Area(D)$ and $\partial D'=\partial D$. Consequently:

\begin{lem}\label{lem:reductions} Let $D\to X^*$ be a disc diagram, then there exist disc diagrams $D'\to X^*$ and $D''\to X^*$ satisfying:
\begin{enumerate}
\item $\partial D = \partial D' = \partial D''$,
\item $D'$ is weakly reduced and $D''$ is reduced,
\item $D'$ is obtained from $D$ after a a finite number of moves of types \eqref{item:red1}-\eqref{item:red5}, and $D''$ is obtained from $D$ after a finite number of moves of type \eqref{item:red0}-\eqref{item:red5}.
\end{enumerate}  
\end{lem}

\begin{rem}
Many theorems about disc diagrams in the literature assume that the disc diagram is reduced or minimal complexity, but it is in fact sufficient to consider weakly reduced diagrams. For example, this is the case with Lemma~\ref{lem:Greendlinger} (the Cubical Greendlinger's Lemma).
\end{rem}

\subsection{Cubical small-cancellation}
We use the convention where $\overline \rho$ denotes the path $\rho$ with the opposite orientation. 
A \emph{grid} is a square disc diagram isometric to the product of two intervals. 
Let $\rho$ and $\eta$ be two combinatorial paths in $\widetilde {X^*}$. We say $\rho$ and $\sigma$  are \emph{parallel} if there exists a grid $E \rightarrow \widetilde {X^*}$ with $\partial E = \mu\rho\overline\nu\overline\eta$, where the dual curves dual to edges of $\rho$, ordered with respect to its orientation, are also dual to edges of $\eta$, ordered with respect to its orientation.
Concretely, if $\rho=e_1\cdots e_k$ and $\sigma=f_1 \cdots f_k$, and  $h(e_i)$ and $h(f_i)$ are the curves dual to $e_i$ and $f_i$ respectively, then $\rho$ is a piece if $h(e_i)=h(f_i)$ for each $i \in \{1, \ldots, k\}$.

An \emph{abstract contiguous cone-piece} of $X^*$ in $Y_i$ is a component of $\widetilde{Y}_i \cap \widetilde{Y}_j$, where
$\widetilde{Y}_i$ is a fixed elevation of $Y_i$ to the universal cover $\widetilde {X}$,  and  either $i \neq j$ or where $i = j$
but $\widetilde{Y}_j \neq \widetilde{Y}_i$.
Each abstract contiguous cone-piece $P$ induces a map $P \rightarrow Y_i$ which is the composition $P \hookrightarrow \widetilde {Y}_i \rightarrow Y_i$, and a \emph{contiguous cone-piece} of $Y_j$ in $Y_i$ is a combinatorial path $\rho \rightarrow P$ in an abstract contiguous cone-piece of $Y_j$ in $Y_i$.
An \emph{abstract contiguous wall-piece} of $X^*$ in $Y_i$ is a component of $\widetilde{Y}_i \cap N(\widetilde{H})$, where $\widetilde{H}$ is a hyperplane that is disjoint from  $\widetilde{Y}_i$.
Each abstract contiguous wall-piece $P$ induces a map $P \rightarrow Y_i$, and a \emph{contiguous wall-piece} of $Y_i$ is a combinatorial path $\rho \rightarrow P$ in an abstract contiguous wall-piece of $Y_i$. 
A \emph{piece} is a path parallel to a contiguous cone-piece or wall-piece. 
\begin{figure}
\includegraphics[scale=0.35]{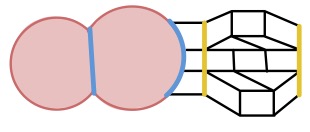}
\caption{Blue paths are contiguous pieces, and yellow paths are pieces but not contiguous pieces.}\label{fig: pieces}
\end{figure}
The difference between contiguous pieces and pieces is illustrated in Figure~\ref{fig: pieces}.

For an integer $p> 0$, we say $X^*$ satisfies the $C(p)$ \emph{small-cancellation} condition if no essential combinatorial closed path in $Y_i$ can be expressed as a concatenation of less than $p$ pieces.
For a constant $\alpha > 0$, we say $X^*$ satisfies the $C'(\alpha)$ \emph{small-cancellation} condition if
$\diam (P) < \alpha \systole{Y_i}$
for every piece $P$ involving $Y_i$.

Note that the $C'(\frac 1 p)$ condition implies the $C(p+1)$ condition. When $p\geq 9$ and $X^*$ is $C(p)$, then each immersion $Y_i \looparrowright X$ lifts to an embedding
  $Y_i \hookrightarrow \widetilde {X^*}$. This is proven in \cite[Thm~4.1]{WiseIsraelHierarchy} for $p\geq 12$,
  and in \cite{Jankiewicz17} for $p\geq 9$.

We record the following observation, a proof of which can be found in~\cite{Arenas2023pi2}.
\begin{lem}\label{lem:bound on size of pieces}
Let $X^* = \langle X \mid Y_1, \dots, Y_m\rangle$ be a cubical presentation where $X$ and  $Y_1, \dots, Y_m$ are compact non-positively curved cube complexes. If $X^*$ satisfies the cubical $C(p)$ condition for $p\geq 2$, then there is a bound on the combinatorial length of pieces of $X^*$.
\end{lem}

\subsection{Greendlinger's Lemma}
A cone-cell $C$ in a disc diagram $D$ is a \emph{boundary cone-cell} if $C$ intersect the boundary $\partial D$ along at least one edge. A non-disconnecting boundary cone-cell $C$ is a \emph{shell of degree $k$} if  $\partial C = RQ$ where 
$Q$ is the maximal subpath of $\partial C$ contained in $\partial D$, and $k$ is the minimal number such that $R$ can be expressed as a concatenation of $k$ pieces. We refer to $R$ as the \emph{innerpath} of $C$ and $Q$ as the \emph{outerpath} of $C$.

A \emph{corner} in a disc diagram $D$ is a vertex $v$ in $\partial D$ of valence $2$ in $D$ that is contained in some square of $D$. 
A \emph{cornsquare} is a square $c$ and  a pair of dual curves emanating from consecutive edges $a, b$ of $c$ that terminate on consecutive edges $a',b'$ of $\partial D$. We abuse the notation and refer to the common vertex of $a',b'$ as a cornsquare as well.
A \emph{spur} is a vertex in $\partial D$ of valence $1$ in $D$. If $D$ contains a spur or a cut-vertex, then $D$ is \emph{singular}.

\begin{defn}[Ladder]\label{defn:ladder}
A \emph{pseudo-grid} between paths $\mu$ and $\nu$ is a square disc diagram $E$ where the boundary path $\partial E$ is a concatenation $\mu \rho \overline{\nu} \overline{\eta}$ such that 
\begin{enumerate}
\item each dual curve starting on $\mu$ ends on $\nu$, and vice versa, 
\item no pair of dual curves starting on $\mu$ cross each other,
\item no pair of dual curves cross each other twice.
\end{enumerate}
If a pseudo-grid $E$ is degenerate then either $\mu = \nu$ or $\rho= \eta$.

A \emph{ladder} is a 
disc diagram $(D, \partial D)\to (X^*, X^{(0)})$ which is an alternating union of cone-cells and/or vertices $C_0, C_2\dots, C_{2n}$ and (possibly degenerate) pseudo-grids $E_1, E_3\dots, E_{2n-1}$, with $n\geq 0$, in the following sense:
\begin{enumerate}
\item the boundary path $\partial D$ is a concatenation $\lambda_1\overline{
\lambda_2}$ where the initial points of $\lambda_1, \lambda_2$ lie in $C_0$, and the terminal points of $\lambda_1,\lambda_2$ lie in $C_{2n}$,
\item $\lambda_1 =\alpha_0 \rho_1\alpha_2\cdots \alpha_{2n-2}\rho_{2n-1}\alpha_{2n}$ and $\lambda_2 =\beta_0\eta_1\beta_2\cdots \beta_{2n-2}\eta_{2n-1}\beta_{2n}$,
\item the boundary path $\partial C_{i} = \nu_{i-1}\alpha_i \overline{\mu_{i+1}}\overline{\beta_i}$ for some $\nu_{i-1}$ and $\mu_{i+1}$ (where $\nu_{-1}$ and $\mu_{2n+1}$ are trivial), and
\item the boundary path $\partial E_{i} = \mu_i \rho_i \overline{\nu_i} \overline{\eta_i}$.
\end{enumerate}
See Figure~\ref{fig: ladder}.
\begin{figure}\begin{tikzpicture}[scale = 0.5]
\node[draw=none,fill=none] (x) at (-1, 0) {$C_0$};
\node[draw=none,fill=none] (x) at (1, 0) {$C_2$};
\node[draw=none,fill=none] (x) at (4.5, 0) {$C_4$};
\node[draw=none,fill=none] (x) at (8, 0) {$C_6$};
\node[draw=none,fill=none] (x) at (10, 0) {$C_8$};
\node[draw=none,fill=none] (x) at (13, 0) {$C_{10}$};
\end{tikzpicture}

\includegraphics[scale=0.35]{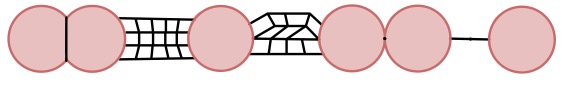}

\begin{tikzpicture}[scale = 0.5]
\node[draw=none,fill=none] (x) at (-0.5, 0) {$E_1$};
\node[draw=none,fill=none] (x) at (2, 0) {$E_3$};
\node[draw=none,fill=none] (x) at (5.5, 0) {$E_5$};
\node[draw=none,fill=none] (x) at (8.5, 0) {$E_7$};
\node[draw=none,fill=none] (x) at (11, 0) {$E_9$};
\node[draw=none,fill=none] (x) at (12, 0) {};
\end{tikzpicture}

\caption{Example of a ladder.}\label{fig: ladder}
\end{figure}
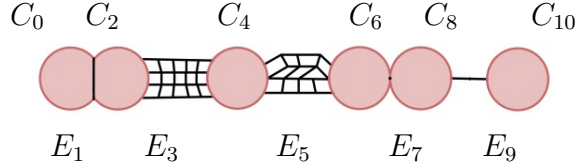
\end{defn}

\begin{lem}[Cubical Greendlinger's Lemma \cite{WiseIsraelHierarchy, Jankiewicz17}]\label{lem:Greendlinger}
Let $X^*= \langle X\mid Y_1, \dots, Y_s\rangle$ be a cubical presentation satisfying the $C(9)$ condition, and let $D\to X^*$ be a weakly reduced disc diagram. Then one of the following holds:
\begin{itemize}
\item $D$ is a ladder, or
\item $D$ has at least three shells of degree $\leq 4$ and/or corners and/or spurs.
\end{itemize}
\end{lem}
We note that our definition of ladder differs slightly from the definitions in  \cite{WiseIsraelHierarchy, Jankiewicz17}, so that a single cone-cell and a single vertex count as ladders here.
Also, the statements in \cite{WiseIsraelHierarchy, Jankiewicz17} assume that the disc diagrams are reduced/minimal complexity, but the proofs work for weakly reduced disc diagrams.

\section{Hyperbolic background}
We explain the convention we will follow. 
A pair $(Y, \dist)$ is a \emph{metric graph}, if there exists a graph $\Gamma$ such that $Y$ is the vertex set of $\Gamma$, and $\dist$ is defined as follows. For each edge of $\Gamma$, we assign a positive number which is the \emph{length} of that edge.
The \emph{length} of a simple path in $\Gamma$ is the sum of the lengths of the edges in the path.
A metric $\dist$ on a set $Y$ is a \emph{graph metric}, if $(Y, \dist)$ is a metric graph. 

In this paper, all edges of metric graphs have one of two lengths: $1$ or $\frac 12$.

 \subsection{Thin bigon criterion for hyperbolicity}
A \emph{bigon} in a geodesic metric space $Y$ is a pair of geodesic segments $\gamma_1, \gamma_2$ in $Y$ with the same endpoints, i.e.\ such that $\gamma_1(0) = \gamma_2(0)$ and $\gamma_1(\ell)= \gamma_2(\ell)$ where $\ell$ is the length of $\gamma$. 
A bigon $\gamma_1, \gamma_2$ is \emph{$\varepsilon$-thin} if $d(\gamma_1(t), \gamma_2(t))<\epsilon$ for all $t\in (0,\ell)$.
If we do not care about the specific value of $\varepsilon$, the above condition is equivalent to the condition that $\im\gamma_1\subseteq N_{\varepsilon'}(\im\gamma_2)$ and $\im\gamma_2\subseteq N_{\varepsilon'}(\im\gamma_1)$ for some $\varepsilon'>0$. 
Indeed, suppose that for every $t\in(0,\ell)$ there exists $t'\in (0,\ell)$ such that $d(\gamma_1(t), \gamma_2(t'))<\varepsilon'$. Then $|t-t'|< \varepsilon$, as otherwise $\gamma_1$ and $\gamma_2$ are not geodesic segments. That implies that $d(\gamma_1(t), \gamma_2(t))\leq d(\gamma_1(t), \gamma_2(t'))+d(\gamma_2(t'),\gamma_2(t)) < 2\varepsilon'$.

This generalizes to paths $\gamma_1, \gamma_2$ whose endpoints are not necessarily the same. We say  $\gamma_1, \gamma_2$ \emph{$\epsilon$-fellow travel} if $d(\gamma_1(t), \gamma_2(t))<\epsilon$ for all $t$.

The following is a hyperbolicity criterion for graphs, due to Papasoglu \cite[Thm 1.4]{Papasoglu95} (see also \cite[Prop 4.6]{WiseIsraelHierarchy}).

\begin{prop}[Thin Bigon Criterion]\label{prop:thin bigon criterion}
Let $Y$ be a graph where all bigons are $\varepsilon$-thin for some $\varepsilon > 0$. Then there exists $\delta = \delta(\varepsilon)$ such that $Y$ is $\delta$-hyperbolic.
\end{prop}

Of course, the converse also holds.

\section{The piece metric}
Let $X^* = \langle X\mid Y_1, \dots, Y_s\rangle$ be a cubical presentation. As explained in Section~\ref{sec:cubical presentations}, we write $X^*$ to denote the complex $X$ with cones over $Y_i$'s attached. In particular, $X$ can be viewed as a subspace of $X^*$. 
The preimage of $X$ in the universal cover $\widetilde{X^*}$ of $X^*$ is denoted by $\widehat X$. Note that $\widehat X$ is a covering space of $X$.
The preimage of the $0$-skeleton of $X$ in $\widetilde{X^*}$ is also the $0$-skeleton of $\widehat X$, so it is denoted by $\widehat X^{(0)}$.

\begin{defn}\label{defn:piece metric}
The \emph{piece length} of a combinatorial path $\gamma$ in $\widehat X^{(0)}$ is the smallest $n$ such that $\gamma = \nu_1\cdots \nu_n$ where each $\nu_k$ is a $1$-cube or a piece.
The \emph{piece metric} $\dist_\p$ on $\widehat X^{(0)}$  is defined as $\dist_\p(a,b) = n$ where $n$ is the smallest piece length of a path from $a$ to $b$.
\end{defn}
We note that $\dist_\p$ is a graph metric when $\widehat X^{(0)}$ is viewed as the graph with all edges of length $1$ obtained from the $1$-skeleton $\widehat{X}^{(1)}$ of $\widehat X$ by adding extra edges between vertices contained in a single piece. We will denote this graph by $(\widehat X^{(0)}, \dist_\p)$.

A \emph{piece decomposition} of a path $\gamma$ is an expression $\gamma= \nu_1 \cdots \nu_k$, where each $\nu_i$ is a piece or $1$-cube.
We make the following easy observation:

\begin{lem}\label{lem: splitting paths}
Let $\gamma, \gamma_1, \gamma_2$ be piece-metric geodesics in $\widehat X^{(0)}$ where $\gamma = \gamma_1\gamma_2$. Then 
$$|\gamma_1|_\p+ |\gamma_2|_\p -1 \leq |\gamma|_\p \leq |\gamma_1|_\p+ |\gamma_2|_\p.$$
\end{lem}

\begin{proof}
Any piece decomposition $\gamma= \nu_1 \cdots \nu_k$ yields piece decompositions of both $\gamma_1$ and $\gamma_2$, where at most one piece $\nu_i$ for $i \in \{1, \ldots, k\}$ further decomposes into the concatenation of $2$ pieces $\nu'_i, \nu''_i$, so $\gamma_1=\nu_1 \cdots \nu'_i$ and $\gamma_2=\nu''_i \cdots \nu_k$. Similarly, any two piece decompositions of $\gamma_1$ and $\gamma_2$ can be concatenated to obtain a piece decomposition of $\gamma$.
\end{proof}

We now prove a few basic facts about the piece metric. First, it is quasi-isometric to the combinatorial metric under fairly weak hypotheses.

\begin{prop}\label{prop: qi}
Let  $X^*  =\langle X \mid Y_1,\dots, Y_s\rangle$ be a cubical presentation satisfying the $C(p)$ condition for $p\geq 2$, and where $X, Y_1, \dots, Y_s$ are compact.
Then $(\widehat X^{(0)}, \dist_\p)$ is quasi-isometric to $(\widehat X^{(0)}, \dist)$ where $\dist$ is the standard combinatorial metric. Moreover, there is a uniform bound on the $\dist_\p$-diameters of cones.
\end{prop}

\begin{proof}
Indeed, $\dist_\p(a,b)\leq \dist(a,b)$ for all $a,b \in \widehat X^{(0)}$, and by Lemma~\ref{lem:bound on size of pieces} there is an upper bound $M$ on the combinatorial length of pieces, so we also have that $\dist(a,b)\leq M \dist_\p(a,b)$.

Since there are only finitely many $Y_i$'s and each $Y_i$ is compact, there must be an upper bound on the diameter of a simple essential curve in $Y_i$ with respect to $\dist$ and thus with respect to $\dist_\p$, which implies the second statement.
\end{proof}

\begin{cor}\label{cor:closeness of geodesics and quasigeodesics}
 Suppose that $\pi_1 X$ is hyperbolic. Let $D\to X^*$ be a square diagram with boundary $\partial D = \gamma \overline\lambda$ where $\gamma$ is a $\dist_\p$-geodesic, and $\lambda$ is a $\dist$-geodesic. Then the bigon $D$ is $M$-thin for a uniform constant $M$.
\end{cor}

\begin{proof}
First note that $D\to X^*$ is a square diagram in $\widehat X$, but it also lifts to $\widetilde X$. The metric $\dist_\p$ also lifts to $\widetilde X^{(0)}$, and by Proposition~\ref{prop: qi} $\dist, \dist_\p$ are quasi-isometric on $\widehat X^{(0)}$, and therefore on $\widetilde X^{(0)}$. The statement then follows from the uniform bound on the Hausdorff distance between geodesics and quasi-geodesics in hyperbolic spaces.
\end{proof}

We note that ladders are thin with respect to the piece metric.

\begin{prop}\label{prop: thin ladder} Suppose that $\pi_1 X$ is hyperbolic.
Let $D\to X^*$ be a ladder with boundary $\partial D = \lambda_1\overline \lambda_2$ as in Definition~\ref{defn:ladder} where each subpath of $\lambda_i$ contained in a single pseudo-grid is a geodesic. 
Then the bigon $\lambda_1, \lambda_2$ is $\epsilon$-thin with respect to $\dist_\p$ for a uniform constant $\epsilon>0$ dependent only on $X^*$.
\end{prop}
\begin{proof}We only show that $\lambda_1\subseteq N_\epsilon(\lambda_2)$, since the argument for $\lambda_2\subseteq N_\epsilon(\lambda_1)$ is analogous. Let $x\in\lambda_1$. We want to show that $\dist_\p(x, \lambda_2)\leq \epsilon$. 
If $x$ belongs to a cone-cell $C$, then by the definition of the ladder, $\lambda_2$ also intersects $C$, so $\dist_\p(x, \lambda_2)$ is bounded by the piece-metric diameter of $C$, which is uniformly bounded by some constant $\epsilon_1$ by Proposition~\ref{prop: qi}.

Otherwise $x$ lies in a pseudo-grid. Let $\rho, \eta$ be subpaths of $\lambda_1, \lambda_2$ respectively, contained in the pseudo-grid which contains $x$. The paths $\rho, \eta$ are both combinatorial geodesics by the assumption. By Proposition~\ref{prop: qi} $\rho, \eta$ start and end at a uniform distance, since they lie in the same cone-cell. By hyperbolicity of $\widetilde X$, there exists $\epsilon_2>0$ such that $\rho, \eta$  $\epsilon_2$-fellow travel. The conclusion follows with $\epsilon= \max\{\epsilon_1, \epsilon_2\}$. 
\end{proof}

In the proof of Theorem~\ref{thm: hyperbolicity of C(p)} we will use the following technical lemma.

\begin{lem}\label{lem: piece length bound}
Let $X^*$ be a cubical presentation, and let $E\to X^*$ be a square diagram, with the induced metric $\dist_\p$. Suppose that $\partial E = \ell Q r\overline \gamma$ where $\gamma$ is a piece geodesic, no dual curve in $E$ crosses $\ell Q r$ twice, and each of $ \ell,  Q, r$ contains no cornsquares of $E$ in its interior.  
Moreover, assume that $|Q|_\p\geq 3$.
Then 
$|\gamma|_\p\geq |\ell|_\p +| Q|_\p + |r|_\p -3$.
\end{lem}

\begin{proof}
See Figure~\ref{fig:piece projections} for a diagram $E$ with $\partial E = \ell Q r\overline \gamma$. 
By the assumptions every dual curve of $E$ starting at $\ell Q r$ must exit the diagram in $\gamma$. Thus each edge of $\ell Q r$ is naturally paired with an edge of $\gamma$. 
For every piece $\nu$ in $\gamma$, we consider all the dual curves $h_1,\dots, h_n$ starting at $\nu$ that exit $E$ in $\ \ell Q r$. 

These define a collection of edges in $\ell Q r$, and every subcollection of such consecutive edges forms a path that is a piece, as it is parallel to some path contained in one of $Y_i$.
By grouping consecutive edges into maximal subpaths contained in one of $\ell$, $ Q$, or $r$, we get pieces $\nu^1, \dots, \nu^k$ whose interiors are pairwise disjoint (ordered consistently with the orientation of $\ell Q r$), and say that $\nu$ \emph{projects} to $\nu^1, \dots, \nu^k$. 

First we claim that each of $\ell,  Q, r$ contains at most one piece $\nu^i$. Suppose to the contrary that $\nu^i, \nu^{i+1}$ are both contained in $\ell$ (and the same argument applies to $ Q, r$). Then each dual curve starting at an edge of $\ell$ lying between $\nu^i$ and $\nu^{i+1}$ must intersect at least one dual curve starting at edges of $\nu^i, \nu^{i+1}$, as otherwise it would also lie in a projection of $\nu$, yielding a cornsquare in $\ell$. 

Thus we can denote the projection of $\nu$ by $\nu^\ell, \nu^{ Q}, \nu^r$ where each piece is a possibly empty projection onto $\ell,  Q, r$ respectively. See left diagram in Figure~\ref{fig:piece projections}. 
\begin{figure}
\includegraphics[scale=0.37]{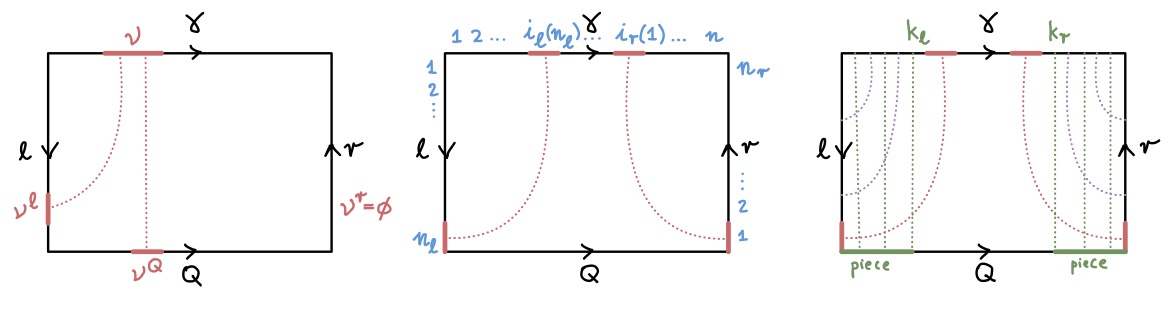}
\caption{Steps of the proof of Lemma~\ref{lem: piece length bound}.}\label{fig:piece projections}
\end{figure}
We will assume that they are oriented consistently with $\ell,  Q, r$ respectively, not necessarily consistently with $\nu$.

Let $\gamma = \nu_1\cdots \nu_n$ be a minimal piece decomposition of $\gamma$, i.e.\ $|\gamma|_\p = n$. 
Let $\ell = \nu^\ell_{i_\ell(1)}\cdots \nu^\ell_{i_\ell(n_{\ell})}$ be the induced piece-decomposition where we only write non-trivial pieces. 
In particular, $i_{\ell}:\{1, \dots, n_{\ell}\}\to \{1, \dots, n\}$ is an injective function. 
We now claim that $i_{\ell}$ is monotone. 
Suppose to the contrary, that $1\leq j<k\leq n_{\ell}$ but $i_{\ell}(j)>i_{\ell}(k)$.
Then there must exists a cornsquare in the connected subpath of $\ell$ containing $\nu^\ell_{i_\ell(k)}$ and $\nu^\ell_{i_\ell(j)}$, which is a contradiction.
Analogously, we get $ Q = \nu^{ Q}_{i_{ Q}(1)}\cdots \nu^{ Q}_{i_{ Q}(n_{ Q})}$ and $r = \nu^r_{i_r(1)}\cdots \nu^r_{i_r(n_{r})}$, and the functions $i_{ Q}, i_{r}$ are monotone. 
These are not necessarily the minimal piece decompositions, but certainly we have $|\ell|_\p \leq n_{\ell}$, $| Q|_\p \leq n_{ Q}$, and $|r|_\p \leq n_{r}$. 
To prove the lemma we will show that  $| Q|_\p + n_{\ell} + n_{r} \leq n+3$.

Note that $i_{\ell}(n_{\ell})$ is the largest index in $\{1, \dots, n\}$ such that $\nu_{i_{\ell}(n_{\ell})}$ has non-trivial projection onto $\ell$, and similarly $i_{r}(1)$ is the lowest index in $\{1, \dots, n\}$ such that $\nu_{i_{r}(1)}$ has non-trivial projection onto $r$. See middle diagram in Figure~\ref{fig:piece projections}. 
 Since $i_{\ell}, i_{r}$ are monotone, $n_{\ell}\leq i_{\ell}(n_{\ell})$ and $n_r\leq n- i_{r}(1)$. Thus, it remains to prove that 
$| Q|_\p\leq i_r(1) - i_{\ell}(n_{\ell}) + 3$.

Let $k_{\ell}$ is the largest number such that $i_{ Q}(k_{\ell})<i_{\ell}(n_{\ell})$. We claim that $\nu_{i_{ Q}(1)}^ Q\cdots \nu_{i_{ Q}(k_{\ell})}^ Q$ is a single piece in $ Q$.
Indeed, the dual curves starting in $\nu_{i_{ Q}(1)}^ Q\dots \nu_{i_{ Q}(k_{\ell})}^ Q$ must all intersect a dual curve starting in $\nu_{i_{\ell}(n_{\ell})}$ and exiting the diagram in $\ell$.
See right diagram in Figure~\ref{fig:piece projections}. 
Similarly, let $k_r$ be the smallest number such that 
$i_{ Q}(k_r)>i_r(1)$ and note that $\nu_{i_{ Q}(k_r)}^ Q\dots \nu_{i_{ Q}(n_ Q)}^ Q$ is a single piece in $ Q$.
 By assumption, $| Q|_p\geq 3$, so the subpath $\nu^{ Q}_{i_{ Q}(k_{\ell}+1)}\cdots \nu^{ Q}_{i_{ Q}(k_r-1)}$ is nonempty. In particular, $k_r> k_{\ell}+1$. 
 Since by definition of $k_\ell, k_r$ we have $i_{ Q}(k_{\ell}+1) \geq i_{\ell}(n_{\ell})$ and $i_{ Q}(k_r-1)\leq i_r(1)$, we conclude that 
 \begin{align*}
 &|\nu^{ Q}_{i_{ Q}(k_{\ell}+1)}\nu^{ Q}_{i_{ Q}(k_{\ell}+2)}\cdots \nu^{ Q}_{i_{ Q}(k_r-1)}|_\p\\ 
\leq\ \  &|\nu_{i_{ Q}(k_{\ell}+1)}\nu_{(i_{ Q}(k_{\ell}+1)+1)} \cdots \nu_{i_{ Q}(k_r-1)}|_\p \\
\leq \ \ &|\nu_{ i_{\ell}(n_{\ell})}\nu_{( i_{\ell}(n_{\ell})+1)}\cdots \nu_{i_r(1)}|_\p\\
\leq \ \ &  i_r(1) - i_{\ell}(n_{\ell})+1.
 \end{align*}
This proves that $| Q|_\p\leq i_r(1) - i_{\ell}(n_{\ell})+3$ and completes the proof.
\end{proof}

\section{Proof of hyperbolicity}
In the proof of the next theorem, we show that, under suitable assumptions, $(\widehat X^{(0)}, \dist_\p)$ is a $\delta$-hyperbolic graph to deduce that $\pi_1 X^*$ is hyperbolic. The basic strategy is similar to \cite[Thm 4.7]{WiseIsraelHierarchy}, but the details in this case are significantly more involved.

\begin{thm}\label{thm: hyperbolicity of C(p)}
Let $X^*= \langle X \mid Y_1,\dots, Y_s\rangle$ be a cubical presentation satisfying the $C(p)$ cubical small-cancellation condition for $p\geq 14$, 
where $X, Y_1, \dots, Y_s$ are compact, and $\pi_1X$ is hyperbolic. 
Then $\pi_1 X^*$ is hyperbolic.
\end{thm}

Before we proceed with the proof of the above theorem we introduce a construction that is used in the proof.
Let $Y \subset X$. The \emph{cubical convex hull} of $Y$ in $X$ is the smallest cubically convex subcomplex of $X$ contained in $Y$. That is, it is the smallest subcomplex $Hull(Y)$ satisfying that whenever a corner of an $n$-cube $c$ with $n \geq 2$ lies in $Hull(Y)$, then $c \subset Hull(Y)$.

\begin{construction}[Square pushes]\label{const:push} 
Let $D$ be a minimal complexity disc diagram, and let $\gamma\rho = \partial D$. Let $\lambda$ be a path with the same endpoints as $\gamma$ and lying in the cubical convex hull of $\gamma$, 
such that $\gamma\overline\lambda$ bounds a disc subdiagram $D_0$ of $D$ of maximal area. In particular, $D_0$ is a square disc diagram, and $D' = D - D_0$ is a disc diagram with $\partial D' = \lambda\rho$, which has no corners contained in the interior of the path $\lambda$. 
The diagram $D_0$ can be obtained via a finite sequence of \emph{square pushes}, i.e. a sequence of subdiagrams 
\[
\gamma = K_0 \subseteq K_1\subseteq \dots \subseteq K_{n-1} \subseteq K_n = D_0
\]
where for each $i = 0, \dots, n-1$ the subdiagram $K_{i+1}$ contains $K_i$ and an additional square $s$ such that at least two consecutive edges of $s$ are contained in $K_i$. Choosing a square $s$ and adding it to $K_i$ to obtain $K_{i+1}$ will be referred to as  \emph{pushing a square}.

Note that the sequence of diagrams $K_0, \ldots K_n$ is indeed finite, as  $Area(K_{i+1})=Area(K_{i})+1$ for each $i$, and thus $Area(D-K_{i+1})=Area(D-K_{i})-1$, so $n-1 \leq Area(D)$.

By construction, every dual curve $h$ in $D_0$ starting in $\gamma'$ must exit in $\gamma$. Indeed, every square $S$ that is being pushed has at least two consecutive edges on $\gamma$ (in the first step) or on some $K_i$ (in general). Thus, the $2$ dual curves emanating from $S$ either directly terminate on $\gamma$ or enter $K_i$, crossing some of the previously added squares. By induction on the area of $K_i$, we can thus conclude that these dual curves terminate on $\gamma$.
\end{construction}

\begin{construction}[Sandwich decomposition of a bigon]\label{const:sandwich}
Let $\gamma_1, \gamma_2$ be paths forming a bigon.
Let $D\to X^*$ be a reduced disc diagram with $\partial D = \gamma_1\overline \gamma_2$. 
We define a decomposition of $D$ into three (possibly singular) subdiagrams $D_1\cup D'\cup D_2$ by applying Construction~\ref{const:push} twice as follows:
\begin{itemize}
\item We first apply it to the subpath $\gamma_1\subseteq \partial D$ to obtain a decomposition $D = D_1\cup D''$ where $\partial D_1 = \gamma_1\overline\lambda_1$ and $\partial D'' =  \lambda_1 \overline\gamma_2$.
\item Then we apply it to the subpath $\gamma_2$ of $\partial D''$ and we obtain a decomposition $D'' = D_2\cup D'$ where $\partial D' =  \lambda_1 \overline\lambda_2$ and $\partial D_2 = \lambda_2\overline\gamma_2$.
\end{itemize}
See Figure~\ref{fig:notation} for an example. We note that $D_1, D_2$ are square diagrams.
\end{construction}

\begin{lem}\label{lem:balanced ladder}
Let $X^*= \langle X \mid Y_1,\dots, Y_s\rangle$ be a cubical presentation satisfying the $C(p)$ cubical small-cancellation condition for $p\geq 14$, 
where $X, Y_1, \dots, Y_s$ are compact, and $\pi_1 X$ is hyperbolic. 

Then for any weakly reduced disc diagram $(D, \partial D)\to (X^*, X)$ with $\partial D = \gamma_1\overline \gamma_2$ where $\gamma_1, \gamma_2$ are $\dist_\p$-geodesics and the subdiagram $D'$ obtained from its sandwich decomposition $D = D_1 \cup D' \cup D_2$ is a 
ladder.
\end{lem}

The idea of the proof is as follows. Proceeding by contradiction, if $D'$ is not a ladder, then $D'$  contains a shell whose outerpath is disjoint from the endpoints $q, q'$ of $\gamma_1, \gamma_2$. Using this shell and Lemma~\ref{lem: piece length bound}, we construct a path $\gamma$ with endpoints $q$ and $q'$ and with shorter piece-length than $\gamma_1$, which contradicts the fact that $\gamma_1$ is a piece geodesic (see Figure~\ref{fig:notation}). 

\begin{proof}
Suppose to the contrary that $D'$ is not a 
ladder. We will derive a contradiction with the fact that $\gamma_1, \gamma_2$ are $\dist_\p$-geodesics. By Lemma~\ref{lem:Greendlinger}, 
$D'$ has at least three exposed cells, i.e.\ shells of degree $\leq 4$, corners and/or spurs. 
Two 
of those exposed cells might contain $q$ and $q'$, but there still must be at least one other exposed cell whose boundary path is disjoint from both $q$ and $q'$. By construction of $D'$ in Construction~\ref{const:sandwich}, there are no corners or spurs contained in the interior of the paths $\gamma_1$ and $\gamma_2$, so we conclude that there must be a shell $S$ of degree $\leq 4$ in $D'$ with the outerpath $Q$ contained in $\gamma_1$ or $\gamma_2$.
Up to switching names of $\gamma_1$ and $\gamma_2$, we can assume that $Q$ is contained in $\gamma_1$. 
Let $R$ denote the innerpath of $S$ in $D'$.

Let $e_\ell$ and $e_r$ be the leftmost (first) and the rightmost (last) edge of $R$, and let $ h_\ell,  h_r$ be their dual curves in $D_1$. By Construction~\ref{const:push} $h_\ell,  h_r$ exit $D_1$ in $\gamma_1$. Let $\gamma'_1$ be the minimal subpath of $\gamma_1$ that contains the edges dual to $h_\ell,  h_r$.

Let $H_{\ell}, H_{r}$ be the hyperplanes of $\widehat X$ extending $h_{\ell}, h_r$ respectively. 
Let $\ell, r$ be combinatorial paths in $D_1$ parallel to $h_\ell,  h_r$ and starting at the two endpoints of the path $Q$, respectively.

Consider a minimal complexity square disc diagram $E$ with boundary $\partial E=\ell Q r \overline \gamma'_1$ where $\ell$ and $r$ are combinatorial paths contained in $N(H_{\ell}), N(H_r)$. In particular, $\ell$ and $r$ do not intersect $H_\ell$ and $H_r$ respectively. Such a diagram $E$ exists since we can choose a subdiagram of $D_1$. Amongst all possible choices of $\ell, r$ and $E$ we pick a diagram with minimal area. A feature of the choice of $E$ is that it has no cornsquares in the interiors of $\ell$ and $r$, as otherwise we could push that cornsquare out and reduce the area. Up to possibly replacing $Q$ with another path with the same endpoints contained in the same cone, we can assume that $Q$ has no cornsquares either. We will assume that this is the case for the remainder of the proof. 

We will be applying Lemma~\ref{lem: piece length bound} to $E$, so we first verify that the assumptions are satisfied. By Lemma~\ref{lem:Greendlinger}, $|Q|_\p \geq p-4> 3$.
Next, we claim that every dual curve starting in $\ell Q r$ exits $E$ in $\gamma'_1$. The cases of dual curves starting in $\ell$ and $r$ are analogous, so we only explain the argument for $\ell$. Consider the subdiagram $E'=E\cup S$ of $D'$. Let $e$ be an edge in $\ell$. 
Note that the dual curve $h$ to $e$ in $E'$ cannot terminate on $Q$, since this would imply that there is a cornsquare on $Q$. If $h$ terminates on $r$, then $h$ is parallel to $Q$, and therefore $Q$ is a single wall-piece, contradicting the $C(p)$ condition.
Thus, $h$ must terminate on $\gamma'_1$. 
Let now $e$ be an edge of $Q$, and $h$ its dual curve in $E$. We already know that $h$ cannot exit $E$ in $\ell$ or $r$. If $h$ exited $E$ in $Q$, it would either yield a cornsquare in the interior of $Q$, contradicting the choice or $Q$, or it would yield a bigon formed from $2$ squares glued along a pair of adjacent edges, contradicting the minimal complexity of $E$.

Since no dual curve in $E$ crosses $\ell Q r$ twice, there are no cornsquares in none of $\ell$, $Q$, and $r$, and $|Q|_\p \geq 3$, Lemma~\ref{lem: piece length bound} implies that $|Q|_\p + |\ell|_\p + |r|_\p  -3 \leq |\gamma'_1|_\p$. Recall that $R$ is the innerpath of the shell $S$ of degree $\leq 4$ in $D'$.
By definition $|R|_\p\leq 4$, and so the $C(p)$ condition with $p\geq 14$ for $X^*$ implies that $|Q|_\p\geq p - 4\geq 10> |R|_\p + 5$. 
Combining the two inequalities we get $$|\overline\ell R r|_\p\leq  |R|_\p + |\ell|_\p + |r|_\p <  (|Q|_\p -5) + |\ell|_\p + |r|_\p \leq |\gamma'_1|_\p-2.$$
In particular, if we write $\gamma_1 = \gamma_{1,\ell}\gamma'_1\gamma_{1,r}$, then using Lemma~\ref{lem: splitting paths}  we get 
$$|\gamma_1|_\p\geq |\gamma_{1,\ell}|_\p+|\gamma'_1|_\p +|\gamma_{1,r}|_\p -2 > |\gamma_{1,\ell}|_\p+|\overline\ell R r|_\p +|\gamma_{1,r}|_\p  \geq  |\gamma_{1,\ell}\overline\ell R r\gamma_{1,r}|_\p $$ which contradicts the fact that $\gamma_1$ is a piece-geodesic, completing the proof.
\begin{figure}
\centerline{\includegraphics[scale=0.5]{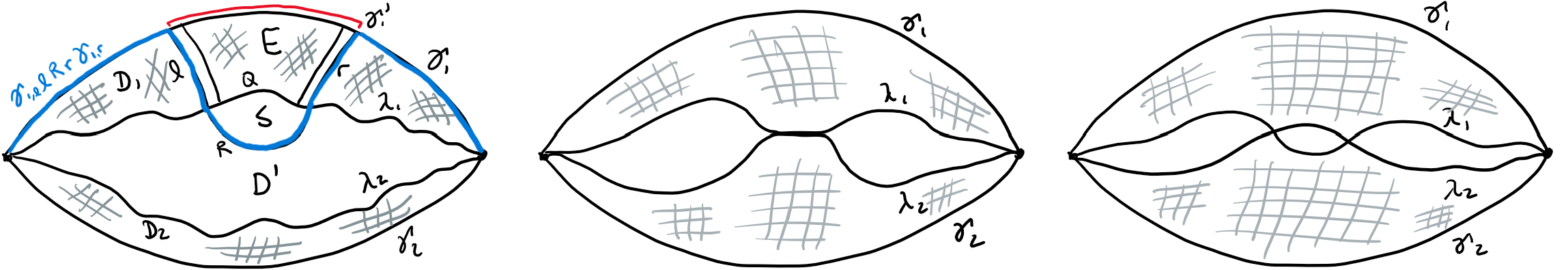}}
\caption{On the left, notation in the proof of Lemma~\ref{lem:balanced ladder}; on the centre, possible overlapping paths in the proof of~\ref{thm: hyperbolicity of C(p)}; on the right, the impossible transversal intersections described in the proof of~\ref{thm: hyperbolicity of C(p)}.}
\label{fig:notation}
\end{figure} 
\end{proof}

We now combine the previous ingredients to finish the proof of Theorem~\ref{thm: hyperbolicity of C(p)}. Here is an outline of the proof. We consider a piece-geodesic bigon and want to prove that it is thin in the piece-metric. We apply the reduction moves from Definition~\ref{defn:reduction moves} to obtain a new bigon. We then apply Lemma~\ref{lem:balanced ladder} to show that the middle layer of its sandwich decomposition is a ladder, and thus it is thin by Proposition~\ref{prop: thin ladder}. Finally, we use Corollary~\ref{cor:closeness of geodesics and quasigeodesics} to deduce that the other layers are also thin, and consequently the original bigon is thin.

\begin{proof}[Proof of Theorem~\ref{thm: hyperbolicity of C(p)}] We prove that the coned-off space $(\widehat X^{(0)}, \dist_\p)$ is $\delta'$-hyperbolic for some $\delta'$ by showing that it satisfies the bigon criterion (Proposition~\ref{prop:thin bigon criterion}). 

Let $\gamma_1, \gamma_2$ be $\dist_\p$-geodesic segments forming a bigon. Pick combinatorial geodesics $\lambda_1, \lambda_2$ such that   $\gamma_1\overline \lambda_1$ and $\gamma_2\overline \lambda_2$ bound square disc diagrams $D_1$ and $D_2$, as they are path-homotopic in $\widehat X$. Let $D'\to X^*$ be a reduced disc diagram with boundary $\partial D' = \lambda_1\overline\lambda_2$. 
By glueing $D_1\cup D' \cup D_2$ along $\lambda_1$ and $\lambda_2$ respectively, we obtain a (possibly non weakly reduced) disc diagram $D\to X^*$ with $\partial D = \gamma_1\overline\gamma_2$.

By Lemma~\ref{lem:reductions}, there exists a sequence of reduction moves (1) - (5) from Definition~\ref{defn:reduction moves} that turns $D$ into a weakly reduced disc diagram. We describe how each reduction move transforms a quintuple $(D, \lambda_1, \lambda_1', \lambda_2, \lambda_2')$ into a quintuple $(\ddot D, \ddot \lambda_1, \ddot \lambda_1', \ddot \lambda_2, \ddot\lambda_2')$, both satisfying conditions below. 
Our sequence starts with $(D, \lambda_1, \lambda_1', \lambda_2, \lambda_2') = (D, \lambda_1, \lambda_1, \lambda_2, \lambda_2)$.
At each step $(D, \lambda_1, \lambda_1', \lambda_2, \lambda_2')$ satisfies:
\begin{enumerate}[(a)]
\item $D$ is a disc diagram with $\partial D = \gamma_1\overline\gamma_2$, 
\item each $\lambda_i$ is an embedded combinatorial path in $D_\square$, such that the three subdiagrams $D_1$, $D_2$ and $D'$ of $D$ with boundaries $\gamma_1\overline\lambda_1$, $\lambda_2\overline\gamma_2$ and $\lambda_1\overline\lambda_2$ respectively,
 are reduced, and $D_1, D_2$ are square diagrams,
 \item each $\lambda_i'$ is an embedded combinatorial path in $D$  (not necessarily in $D_\square$) such that $\lambda_i'$ is a combinatorial geodesic in $\widetilde X^*$, and $\lambda_i'\cap D_\square = \lambda_i \cap D_\square$,
\end{enumerate}
and after applying a reduction move, we obtain a new quintuple $(\ddot D, \ddot \lambda_1, \ddot \lambda_1', \ddot \lambda_2, \ddot\lambda_2')$ satisfying the above conditions.

Since the subdiagrams $D_1, D_2$ are square disc diagrams and $D'$ is reduced, we never apply the reduction moves~\eqref{item:red2} and~\eqref{item:red3}, since those would have to be performed within $D'$, contradicting that $D'$ is reduced.

We now describe the transformation from $\lambda_i$ to $\ddot\lambda_i$ and from $\lambda_i'$ to $\ddot\lambda_i'$, for each reduction move. In each case the change will occur only within a subdiagram $B$ that is transformed to $\ddot B$ by the reduction.  See Figure~\ref{fig:bigon reduction}. We assume $\lambda_i$ intersects the interior of $B$. We note that it might happen that both $\lambda_1, \lambda_2$ intersect $B$, in which case we apply the transformations to both $\lambda_1, \lambda_1'$ and $\lambda_2, \lambda_2'$, according to the rules described below. It might happen that paths $\lambda_1, \lambda_2$ overlap, but at no step they intersect transversally, i.e.\ at each step they yield a decomposition of the diagram $D$ into $D_1\cup D'\cup D_2$. See Figure~\ref{fig:notation}.

\begin{figure}
\includegraphics[scale=0.7]{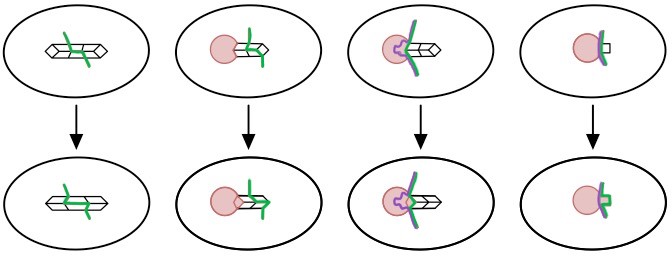}
\caption{Reduction moves~\eqref{item:red1}, ~\eqref{item:red4}, ~\eqref{item:red4}, ~\eqref{item:red5}. In the first two cases, the path $\lambda_i = \lambda_i'$ is reduced to the path $\ddot\lambda_i = \ddot\lambda_i'$. In the last two case, $\ddot\lambda_i'=\lambda_i'$ (purple paths), but $\ddot \lambda_i$ differs from $\lambda_i$.}\label{fig:bigon reduction}
\end{figure}

Consider reduction move~\eqref{item:red1}. Let $B$ be a bigon-subdiagram of $D$, which is to be reduced. Since this reduction move involves only squares, we have ${\lambda_i}\cap{B} = {\lambda_i'}\cap{B}$. Note that  ${\lambda_i}\cap{B}$ cannot join the two corners of $B$ since $\lambda_i'$ is a combinatorial geodesic. Thus ${\lambda_i}\cap{B}$ must be a combinatorial geodesic crossing both dual curves associated to $B$. 
For each $\lambda_i\cap{B}$ we set $\ddot\lambda_i\cap {\ddot B} = \ddot \lambda_i'\cap {\ddot B}$ to the combinatorial geodesic with the same endpoints in $\ddot B$ and maximizing the area of $D_i$. See the first diagram in Figure~\ref{fig:bigon reduction}.
Thus the reduction move yields a new disc diagram $\ddot D$ with new paths $\ddot\lambda_i, \ddot\lambda_i'$ for $i=1,2$ satisfying the required conditions. 
We note that in the case where both $\lambda_1, \lambda_2$ intersect $B$, the choices of $\ddot\lambda_1, \ddot \lambda_2$ ensure that $\ddot\lambda_1, \ddot \lambda_2$ do not intersect transversally. 

Consider reduction move~\eqref{item:red4}. Let $B$ be a subdiagram associated to a cornsquare $s$ and its dual curves ending on a cone-cell $C$.
 We set $\ddot \lambda_i\cap {\ddot B}$ to the combinatorial path with the same length and endpoints in $\ddot B$ and maximizing the area of $D_i$.
 If $\lambda_i'$ coincides with $\lambda_i$ in $B$ we set $\ddot \lambda_i'\cap{\ddot B} = \ddot \lambda_i\cap{\ddot B}$. See the second diagram in Figure~\ref{fig:bigon reduction}.
Otherwise, we set $\ddot \lambda_i'\cap{\ddot B} = \lambda_i'\cap{B}$.  See the third diagram in Figure~\ref{fig:bigon reduction}. Again, in the case where both $\lambda_1, \lambda_2$ intersect $B$, the choices of transformed paths ensure that no transversal intersection occurs. 

Finally, consider reduction move~\eqref{item:red5}. Let $B$ be a subdiagram consisting of a square $s$ overlapping with a cone-cell $C$ along a single edge $e$.  Thus $\lambda_i\cap {B} = e$.  
We set $\ddot \lambda_i\cap{\ddot B}$ to the path $\partial s-e$.
We also set ${\ddot \lambda_i'}\cap\ddot B = {\lambda_i'}\cap B$. 
 See the last diagram in Figure~\ref{fig:bigon reduction}.

\textbf{Working under the  assumption of weakly reduced:} We now assume that  $(D,\lambda_1, \lambda_1',\lambda_2, \lambda_2')$ satisfies conditions (a)-(c) above and that $D$ is weakly reduced. Following the notation in (b),
we claim that either $D_1\cup D' \cup D_2$ is the sandwich decomposition of $D$, or  we can push squares into $D_1$ and $D_2$ modifying $\lambda_i, \lambda_i'$ while preserving conditions (a)-(c). Since $\lambda_i \cap  D_\square=\lambda'_i \cap D_\square$ and $\lambda_i'$ is a geodesic in $\widetilde{X^*}$,  no square $s$ in $ D'_\square$ has three sides on $\lambda_i$.  So if  a square $s$ on $\lambda_i$  can be pushed into $D_i$, then $s$ must have two consecutive edges $a,b$ on $\lambda_i$.
 Let $\lambda_i=\ell_iabr_i$ where $\ell_i, r_i$ are  subpaths of $\lambda_i-ab$. Likewise, let  $\lambda'_i=\ell'_iabr'_i$. Finally, define $\ddot \lambda_i= \ell_icd  r_i$ and $\ddot \lambda'_i= \ell'_icd  r'_i$ where $c,d$ are the other two edges of $s$. The quintuple $(D,\ddot \lambda_1, \ddot \lambda_1',\ddot \lambda_2, \ddot\lambda_2')$ satisfies conditions (a)-(c). Indeed, since  $|\lambda'_i|=|\ddot \lambda'_i|$, condition (c) is still satisfied. As this replacement does not affect $D$, nor the property of being (weakly) reduced, nor that $D_1$ and $D_2$ are square diagrams,  conditions (a) and (b) are also preserved. We arrive at the  sandwich decomposition after finitely many square-pushes.

\textbf{Working under the further assumption that $D_1\cup D' \cup D_2$ is the sandwich decomposition of $D$:} For the remainder of the proof, we assume   that  $(D,\lambda_1, \lambda_1',\lambda_2, \lambda_2')$ satisfies conditions (a)-(c)  and that $D$ is weakly reduced, and that the associated $D_1\cup D' \cup D_2$ is the sandwich decomposition of $D$.
By Lemma~\ref{lem:balanced ladder}, the subdiagram $D'$ with $\partial D' = \lambda_1\overline\lambda_2$ is a 
ladder. Let $D''$ be the subdiagram of $D'$ with $\partial D''=\lambda_1'\overline\lambda_2'$. Then $D''$ is also a ladder, since $\lambda_i\cap  D'=\lambda'_i\cap D'$. By Proposition~\ref{prop: thin ladder}, the bigon  $\lambda'_1, \lambda'_2$ is $\epsilon$-thin for a uniform constant $\epsilon$. 
By Corollary~\ref{cor:closeness of geodesics and quasigeodesics} there exists a uniform constant $M$ such that $\gamma_i, \lambda'_i$ is $M$-thin. Consequently,  $\gamma_1, \gamma_2$ is $(\epsilon + 2M)$-thin.
\end{proof}

\bibliographystyle{alpha}
\bibliography{bib9.bib}

%
%
\end{document}